\documentclass{amsart}

\usepackage[utf8]{inputenc}
\usepackage{color}
\usepackage{xcolor}
\usepackage{hyperref}
\usepackage{amssymb}
\usepackage{float}
\usepackage{tikz,tikz-cd}
\usepackage{forest}
\usepackage{nicefrac}
\usepackage{enumerate}
\usepackage[nameinlink, capitalize, noabbrev]{cleveref}
\usetikzlibrary{shapes.geometric,decorations.markings}
\usetikzlibrary{decorations.pathreplacing}
\usetikzlibrary{fit}
\usetikzlibrary{automata}
\usetikzlibrary{positioning}
\usetikzlibrary{intersections}
\tikzset{mytext/.style={font=\small, text=black}}

\hypersetup{
    colorlinks=true,
    linkcolor=teal,
    citecolor=magenta,
    }

\newtheorem{Theorem}{Theorem}

\newtheorem{Conjecture}{Conjecture}
\newtheorem{Corollary}[Conjecture]{Corollary}
\newtheorem{proposition}{Proposition}[section]
\newtheorem{lemma}[proposition]{Lemma}

\newtheorem{theorem}[proposition]{Theorem}

\newtheorem{Question}[Conjecture]{Question}
\theoremstyle{definition}
\newtheorem{remark}[proposition]{Remark}

\numberwithin{equation}{section}

\title{On a question of Abért and Virág}
\author{Jorge Fariña-Asategui}
\address{Jorge Fariña-Asategui: Centre for Mathematical Sciences, Lund University, 223 62 Lund, Sweden -- Department of Mathematics, University of the Basque Country UPV/EHU, 48080 Bilbao, Spain}
\email{jorge.farina\_asategui@math.lu.se}
\thanks{The author is supported by the Spanish Government, grant PID2020-117281GB-I00, partly with FEDER funds and by the Walter Gyllenberg Foundation from the Royal Physiographic Society of Lund.}

\begin{document}

\begin{abstract}
Abért and Virág proved in 2005 that the Hausdorff dimension of a non-trivial normal subgroup of a level-transitive 1-dimensional subgroup of the group of $p$-adic automorphisms $W_p$ is always 1. They further asked whether the same holds replacing 1-dimensional with positive dimensional. 

On the one hand, we provide a negative answer in general by giving counterexamples where the non-trivial normal subgroups are not all 1-dimensional. Furthermore, these counterexamples are pro-$p$ subgroups of $W_p$ with positive Hausdorff dimension in $W_p$ but with non-trivial center, and thus not weakly branch.

On the other hand, we restrict ourselves to the class of self-similar groups and answer the question of Abért and Virág in the positive in this case. Along the way, we generalize a result of Abért and Virág on the closed subgroups of $W_p$ being perfect in the sense of Hausdorff dimension to closed subgroups of any iterated wreath product $W_H$ and show that self-similar positive-dimensional subgroups of $W_H$ do not satisfy any group law.
\end{abstract}

\maketitle

\section{introduction}
\label{section: introduction}

Many important classes of groups arise as subgroups of iterated wreath products. Remarkable examples are both the class of branch groups \cite{Handbook}, introduced by Grigorchuk in 1997 and first appeared in print in \cite{NewHorizonsGrigorchuk}, which includes the first examples of groups of intermediate growth (so also of amenable but not elementary amenable groups)~\cite{GrigorchukMilnor}; and the class of self-similar groups \cite{SelfSimilar}, which includes the first example of an amenable but not subexponetially amenable group \cite{BartholdiVirag} and iterated monodromy groups which are important invariants in complex dynamics \cite{Thurston}.

Many properties of iterated wreath products and their subgroups were established by the remarkable work of Abért and Virág in \cite{AbertVirag}. Among other results, they proved the following; see \cref{section: Preliminaries} for the unexplained terms and notation in the following result and elsewhere in the introduction:

\begin{theorem}[{see {\cite[Theorem 7]{AbertVirag}}}]
    \label{theorem: abert and virag normal}
    Let $G\le_c W_p$ be a level-transitive closed subgroup with Hausdorff dimension 1 in $W_p$. Then any non-trivial normal closed subgroup $N\trianglelefteq_c G$ has Hausdorff dimension 1 in $G$.
\end{theorem}

Right after stating \cref{theorem: abert and virag normal} in \cite{AbertVirag}, they asked the following question:

\begin{Question}[Abért and Virág, 2005]
    \label{Question: abert and virag normal}
    Does \cref{theorem: abert and virag normal} hold for $G\le_c W_p$ level-transitive with positive Hausdorff dimension in $W_p$?
\end{Question}

Furthermore, the proof of \cref{theorem: abert and virag normal} in \cite{AbertVirag} shows that 1-dimensional subgroups of $W_p$ are weakly branch. In view of this, it is natural to further ask the following:
\begin{Question}
    \label{Question: positive dim implies weakly branch}
    If $G\le_c W_p$ is level-transitive with positive Hausdorff dimension in~$W_p$, is then $G$ weakly branch?
\end{Question}

Our first result in this paper answers both \textcolor{teal}{Questions} \ref{Question: abert and virag normal} and \ref{Question: positive dim implies weakly branch} in the negative:

\begin{Theorem}
    \label{Theorem: counterexample}
    For any $m\ge 2$, any transitive $H\le \mathrm{Sym}(m)$ and any $\alpha\in [0,1/m]$, there exists a closed subgroup $G\le_c W_H$ such that:
    \begin{enumerate}[\normalfont(i)]
        \item $G$ is level-transitive;
        \item $G$ has strong Hausdorff dimension $\alpha$ in $W_H$;
        \item there is a non-trivial normal closed subgroup $N\trianglelefteq_c G$ with $\mathrm{hdim}_G(N)=0$;
        \item $G$ is not weakly branch.
    \end{enumerate}
\end{Theorem}

Each group in \cref{Theorem: counterexample} is constructed from another subgroup $K\le_c W_H$, and it inherits some of its properties. If $K$ is branch then the corresponding subgroup~$G_K$, even if it is not branch in the usual geometrical sense, it is branch according to the more general algebraic definition in \cite[Definition 1.1]{Handbook}. Indeed, the canonical action of $G_K$ on the tree determined by its branch structure as a group of tree automorphisms (inherited by a branch structure in $K$) is equivalent to the action of $K$ on the original tree; see the discussion at the end of \cref{section: general setting}. The following question was asked by Bartholdi, Grigorchuk and \v{S}uni\'{c} in \cite{Handbook}:

\begin{Question}[{see {\cite[Section 1.3, page 1012]{Handbook}}}]
    \label{Question: question 2 handbook}
    Every branch group from \cite[Definition 1.1]{Handbook} acts canonically on the tree determined by its branch structure as a group of tree automorphisms. Is the kernel of this action necessarily central?
\end{Question}

For the group $G_K$ above, the kernel of this canonical action is precisely the subgroup of rooted automorphisms $H$. Therefore, if $H$ is not abelian and $K$ is branch then $Z(G_K)=Z(H)<H$ (see \cref{proposition: center}) but the kernel of this action is $H\ne 1$. Hence the kernel of the action is not central, answering \cref{Question: question 2 handbook} in the negative.

Despite \cref{Theorem: counterexample}, one may wonder whether \textcolor{teal}{Questions} \ref{Question: abert and virag normal} and \ref{Question: positive dim implies weakly branch} have a positive answer if one restricts to a better-behaved subfamily of groups acting on rooted trees. The family of self-similar groups is a good candidate. In fact, the author proved in \cite{RestrictedSpectra} that self-similar groups are well behaved in terms of Hausdorff dimension as they have strong Hausdorff dimension, i.e. their Hausdorff dimension is given by a proper limit. Note that this is also the case for 1-dimensional subgroups. Thus, it is natural to consider \textcolor{teal}{Questions} \ref{Question: abert and virag normal} and \ref{Question: positive dim implies weakly branch} restricted to self-similar positive-dimensional subgroups of $W_p$, and more generally of $W_H$. 

A key tool in the proof of \cref{theorem: abert and virag normal}, is that closed subgroups of $W_p$ are perfect in the sense of Hausdorff dimension:

\begin{theorem}[{see {\cite[Theorem 5]{AbertVirag}}}]
    \label{theorem: perfect hdim}
    Let $G\le_c W_p$ be a closed subgroup. Then 
    $$\mathrm{hdim}_{W_p}(G)=\mathrm{hdim}_{W_p}(G').$$
\end{theorem}

As a first step towards a positive answer to \textcolor{teal}{Questions} \ref{Question: abert and virag normal} and \ref{Question: positive dim implies weakly branch}, we generalize the proof of \cref{theorem: perfect hdim} to any iterated wreath product:

\begin{Theorem}
\label{theorem: hdim of commutators}
    Let $G\le_c W_H$ be a closed subgroup. Then 
    $$\mathrm{hdim}_{W_H}(G)=\mathrm{hdim}_{W_H}(G').$$
\end{Theorem}

Note that \cref{theorem: hdim of commutators} further generalizes \cite[Proposition 3.1]{JorgeMikel} from topologically finitely generated subgroups to arbitrary subgroups. The proof of \cref{theorem: hdim of commutators} also fixes a gap in the original proof of \cref{theorem: perfect hdim}, where the authors claimed that a result of Kovács and Newman in \cite{KovacsNewman} gives a certain bound in their proof but this bound was in fact a conjecture at the time \cite{AbertVirag} was published; see \cref{remark: av gap}.

In any case, we simplify the proof of \cref{theorem: perfect hdim} and show that the simpler bound in \cite{KovacsPraeger} (see also \cite{Bound}) is enough to prove perfectness of the  Hausdorff dimension, even in the general setting of \cref{theorem: hdim of commutators}.

In order to answer \textcolor{teal}{Questions} \ref{Question: abert and virag normal} and \ref{Question: positive dim implies weakly branch} in the positive for self-similar positive-dimensional groups, we prove that they have large rigid stabilizers; see \textcolor{teal}{Lemmata}~\ref{lemma: rists dimension} and \ref{lemma: rists lower bound for strong hdim}. This together with \cref{theorem: hdim of commutators} yields a positive answer to both questions in this case; see \cref{proposition: self-similar}. 

We take a step further and apply recent results in \cite{FiniteSpectrum} to further obtain the full Hausdorff spectrum of these groups. We summarize all these results in the following theorem:

\begin{Theorem}
    \label{Theorem: self-similar spectra}
    Let $G\le_c W_H$ be a closed self-similar level-transitive subgroup with positive Hausdorff dimension in $W_H$. Then:
    \begin{enumerate}[\normalfont(i)]
        \item $G$ is weakly branch;
        \item the full Hausdorff spectrum is complete: $\mathrm{hspec}(G)=[0,1]$;
        \item any non-trivial normal closed subgroup $1\ne N\trianglelefteq_c G$ has Hausdorff dimension~1 in $G$. In particular, the normal Hausdorff spectrum of $G$ is given by $\mathrm{hspec}_\trianglelefteq(G)=\{0,1\}$.
    \end{enumerate}
\end{Theorem}

In particular, from \textcolor{teal}{Theorems} \ref{theorem: hdim of commutators} and \ref{Theorem: self-similar spectra}, we get a restriction on which groups can be realized as self-similar level-transitive groups with positive Hausdorff dimension:

\begin{Corollary}
    Let $G$ be a profinite group with non-trivial center. Then $G$ does not admit a positive-dimensional faithful action on a regular rooted tree as a self-similar level-transitive closed group.
\end{Corollary}

Furthermore, note that as a consequence of \cref{theorem: hdim of commutators}, any virtually solvable subgroup $G\le W_H$ has zero Hausdorff dimension in $W_H$. The analogous result for $W_p$ was deduced by Abért and Virág using \cref{theorem: perfect hdim} and they conjectured that the same holds for groups satisfying a group law, i.e. groups satisfying a group law should have Hausdorff dimension zero in $W_p$. Since weakly branch closed subgroups of $W_H$ do not satisfy any group law by a result of Abért \cite[Theorem 1.1]{Abert}, \cref{Theorem: self-similar spectra} also answers this conjecture of Abért and Virág in the self-similar case:

\begin{Corollary}
    Let $G\le W_H$ be self-similar. Then, if $G$ satisfies a group law it has Hausdorff dimension zero in $W_H$.
\end{Corollary}

As a weakly regular branch profinite group $G\le W_H$ has positive Hausdorff dimension in $W_H$ by \cite[Proposition 5.3]{RestrictedSpectra} and it is level-transitive by definition, \cref{Theorem: self-similar spectra} yields the full and the normal Hausdorff spectra of any weakly regular branch group, answering partially \cite[Question 3.7]{FiniteSpectrum}:

\begin{Corollary}
\label{Corollary: spectra of weakly regular branch groups}
    Let $G\le_c W_H$ be a weakly regular branch closed subgroup. Then $G$ satisfies the assumption in \cref{Theorem: self-similar spectra}.
\end{Corollary}

In particular, \cref{Corollary: spectra of weakly regular branch groups} implies that in order to explicitly compute the Hausdorff dimension of a weakly regular branch group $G\le_c W_H$ in $W_H$, one may instead compute the Hausdorff dimension of a non-trivial normal branching subgroup $N\trianglelefteq G$ in $W_H$, if $G$ has such a subgroup; for example if $G$ is fractal \cite[Lemma~1.3]{Bartholdi-Siegenthaler-Zalesskii}. This is usually technically less involved as it requires less computations; see \cite{JorgeMikel} for an example.

Finally, we note that the proof of the key lemma (\cref{lemma: rists dimension}) used in the proof of \cref{Theorem: self-similar spectra}  fails as soon as one drops the self-similarity assumption for groups with dimension strictly smaller than 1. However, for 1-dimensional subgroups one may drop the self-similarity assumption in the proof and then it reduces precisely to the proof of Abért and Virág in \cite[Lemma 9.4]{AbertVirag}.

\subsection*{\textit{\textmd{Organization}}} In \cref{section: Preliminaries} we introduce the basic concepts to follow the subsequent sections. We prove \cref{Theorem: counterexample} in \cref{section: general setting}, \cref{theorem: hdim of commutators} in \cref{section: perfect hdim} and \cref{Theorem: self-similar spectra} in  \cref{section: self-similar setting}.

\subsection*{\textit{\textmd{Notation}}} For a set $S$, we denote its cardinality by $\#S$. The positive integer $m\ge 2$ is used to denote the degree of a regular rooted tree $T$ throughout the text, and when it is prime we shall use $p$ instead. Logarithms will be taken in base $m$ if not stated otherwise. We write $H\trianglelefteq G$, $H\le_c G$ and $H\trianglelefteq_c G$ to denote a normal, a closed and a normal closed subgroup of $G$ respectively.

\subsection*{Acknowledgements} 

I would like to thank Oihana Garaialde Ocaña and Jone Uria-Albizuri for fruitful discussions on \cref{theorem: abert and virag normal} while working on \cite{FiniteSpectrum} and to Anitha Thillaisundaram for useful feedback on a first version of this manuscript.

\section{Preliminaries}
\label{section: Preliminaries}

\subsection{Groups acting on rooted trees}

Let us fix an integer $m\ge 2$. The \textit{$m$-adic tree} $T$ is the infinite rooted tree such that every vertex has exactly $m$ descendants. For each $n\ge 1$, the set of vertices at distance $n$ from the root $\emptyset$ form the \textit{$n$th level of the tree} $\mathcal{L}_n$, and the set of vertices at distance at most $n$ from the root form the \textit{nth truncated tree} $T^{[n]}$. We shall use the word level to refer to the integer $n$ too. The $m$-adic tree $T$ may be identified with the free monoid on the set $\{1,\dotsc,m\}$, and each vertex $v\in T$ can be written as a finite word $v=x_1\dotsb x_{l(v)}$ with $x_i\in \{1,\dotsc,m\}$, where $l(v)$ denotes the level at which $v$ lies.

The group of graph automorphisms of $T$ will be denoted by $\mathrm{Aut}~T$. Note that elements in $\mathrm{Aut}~T$ fix the root and permute vertices at the same level. For $n\ge 1$, we denote by $\mathrm{St}(n)$ the pointwise stabilizer of the $n$th level of the tree, i.e. those automorphisms of $T$ fixing every vertex at the $n$th level.

For any $1\le n\le \infty$, any vertex $v$ and any $g\in \mathrm{Aut}~T$, we define the \textit{section of $g$ at $v$ of depth $n$} as the unique automorphism $g|_v^n \in\mathrm{Aut}~T^{[n]}$ such that 
$$(vu)^g=v^gu^{g|_v},$$
for every $u\in T^{[n]}$. For $n=\infty$ we write $g|_v$ for $g|_v^\infty$.

Given a subgroup $H\le \mathrm{Sym}(m)$ we define the \textit{iterated wreath product} $W_H$ as
$$W_H:=\{g\in \mathrm{Aut}~T\mid g|_v^1\in H\text{ for all }v\in T\}.$$
In particular $\mathrm{Aut}~T=W_{\mathrm{Sym}(m)}$. If $m=p$ a prime we shall write $W_p$ for the iterated wreath product corresponding to the subgroup $H$ generated by the $p$-cycle $(1\,\dotsb \,p)$ and call it the \textit{group of $p$-adic automorphisms}. Note that in general $H\le W_H$, where we let $H$ act on $T$ via \textit{rooted automorphisms}, i.e. each $h\in H$ acts on $T$ via the automorphism $g$ given by $g|_\emptyset^1=h$ and $g|_v^1=1$ for $v\in T\setminus\{\emptyset\}$.

For any $n\ge 1$ we define the monomorphisms $\psi_n:\mathrm{St}(n)\to \mathrm{Aut}~T\times\overset{m^n}{\ldots}\times \mathrm{Aut}~T$ via
$$g\mapsto (g|_{v_1},\dotsc,g|_{v_{m^n}}).$$
We shall write $\psi$ for $\psi_1$.

Let us fix a subgroup $G\le \mathrm{Aut}~T$. We define the level-stabilizers of $G$ in the obvious way: for every $n\ge 1$ we set $\mathrm{St}_G(n):=G\cap \mathrm{St}(n)$. We say that $G$ is \textit{level-transitive} if $G$ acts transitively on every level of $T$. We say that $G$ is \textit{self-similar} if for every $g\in G$ and $v\in T$ we have $g|_v\in G$. If $G$ is self-similar then the monomorphism $\psi_n$ induces a monomorphism $\psi_n:\mathrm{St}_G(n)\to G\times\overset{m^n}{\ldots}\times G$ for every $n\ge 1$.

For a vertex $v\in T$ we define the corresponding \textit{rigid vertex stabilizer} $\mathrm{rist}_G(v)$ as the subgroup of $G$ consisting of automorphisms fixing $v$ and every vertex not contained in the subtree rooted at $v$. We further define for any $n\ge 1$ the corresponding \textit{rigid level stabilizer} $\mathrm{Rist}_G(n)$ as the direct product of the rigid vertex stabilizers of vertices at the $n$th level. We say a level-transitive group $G$ is \textit{weakly branch} if for every $n\ge 1$ the rigid level stabilizer $\mathrm{Rist}_G(n)$ is non-trivial and \textit{branch} if $\mathrm{Rist}_G(n)$ is of finite index in $G$.

\subsection{Hausdorff dimension}

Let $G$ be a countably based profinite group with a filtration of open normal subgroups
$$G=G_0\ge G_1\ge G_2\ge \dotsb \ge G_n\ge \dotsb \ge \bigcap_{n\ge 0} G_n=1.$$
Then $G$ may be endowed with the metric $d:G\times G\to [0,\infty)$ given by
$$d(g,h):=\inf_{n\ge 0}\{|G:G_n|^{-1}\mid gh^{-1}\in G_n\}$$
and one may define a Hausdorff dimension for the Borel subsets of $G$. If $H$ is a closed subgroup of $G$ then its Hausdorff dimension coincides with its lower box dimension by \cite[Theorem 2.4]{BarneaShalev}, i.e.
$$\mathrm{hdim}_G^{\{G_n\}_{n\ge 0}}(H)=\liminf_{n\to\infty}\frac{\log|HG_n:G_n|}{\log|G:G_n|}=\liminf_{n\to\infty}\frac{\log|H:H\cap G_n|}{\log|G:G_n|}.$$
We say that $H$ has \textit{strong Hausdorff dimension in $G$} if the lower limit above is a proper limit.

The following lemma first appeared in \cite{PadicAnalytic}. The lemma in this form together with a brief easy proof may be found in \cite{FiniteSpectrum}.

\begin{lemma}[{see {\cite[Lemma 2.4]{FiniteSpectrum}}}]
    \label{lemma: strong dimension then product}
    Let $G$ be a countably based profinite group with closed subgroups $L\le_c K\le_c G$. Consider a filtration series  $\{G_n\}_{n\ge 1}$ and the induced one in $K$, i.e. $\{K_n\}_{n\geq 1}$ where $K_n=K\cap G_n$. If either $L$ has strong Hausdorff dimension in $K$ or $K$ has strong Hausdorff dimension in $G$, then
    $$\mathrm{hdim}_G^{\{G_n\}_{n\ge 1}}(L)=\mathrm{hdim}_G^{\{G_n\}_{n\ge 1}}(K)\cdot \mathrm{hdim}_K^{\{K_n\}_{n\ge 1}}(L).$$
\end{lemma}

Now, the full automorphism group $\mathrm{Aut}~T$, and more generally any iterated wreath product $W_H$, is a countably based profinite group with respect to the level-stabilizer filtration. Therefore, for any closed subgroup $G\le_c W_H$ we may compute its Hausdorff dimension in $W_H$ as
$$\mathrm{hdim}_{W_H}(G)=\liminf_{n\to\infty}\frac{\log|G:\mathrm{St}_G(n)|}{\log|W_H:\mathrm{St}_{W_H}(n)|},$$
where we drop the filtration from the notation as we shall always consider the level-stabilizer filtration. The following result was proved by the author in \cite{RestrictedSpectra}:

\begin{theorem}[{see {\cite[Theorem B and Proposition 1.1]{RestrictedSpectra}}}]
    \label{theorem: strong self-similar}
    Let $G\le_c W_H$ be self-similar. Then $G$ has strong Hausdorff dimension in $W_H$.
\end{theorem}

\section{A counterexample in the general setting}
\label{section: general setting}

We provide a useful construction proving \cref{Theorem: counterexample} and discuss its connection to \cref{Question: question 2 handbook}. We first fix some notation. Let $K\le W_H$. Then we define the subgroup $D_m(K)\le W_H$ as the preimage of the diagonal embedding of $K$ via $\psi$, i.e. $\psi(D_m(K))$ is the diagonal embedding of $K$ in $W_H\times \overset{m}{\ldots}\times W_H$.

\subsection{Proof of \cref{Theorem: counterexample}}

Let us fix $m\ge 2$ and $H\le \mathrm{Sym}(m)$ transitive for the remainder of the section. Let $K\le_c W_H$ be any level-transitive closed subgroup with positive Hausdorff dimension and let $H\le W_H$ act on $T$ via rooted automorphisms. Then we define the group $G_K\le_c W_H$ as
$$G_K:=\langle H,D_m(K)\rangle.$$
Let us see that each group $G$ in the statement of \cref{Theorem: counterexample} may be chosen among the groups $G_K$. Clearly $G_K$ is level-transitive by \cite[Proposition 4.2]{RestrictedSpectra}. The Hausdorff dimension of $G_K$ in $W_H$ is precisely $d/m$, where $d:=\mathrm{hdim}_{W_H}(K)$. Indeed
$$\log|\mathrm{St}_{G_K}(1):\mathrm{St}_{G_K}(n)|=\log|K:\mathrm{St}_K(n-1)|$$
for every $n\ge 1$ as $\mathrm{St}_{G_K}(1)=D_m(K)$, and thus
$$\frac{\log|G_K:\mathrm{St}_{G_K}(n)|}{\log|W_H:\mathrm{St}_{W_H}(n)|}=\frac{\log|H|+\log|K:\mathrm{St}_K(n-1)|}{\log|H|+m\log|W_H:\mathrm{St}_{W_H}(n-1)|}.$$
Hence $\mathrm{hdim}_{W_H}(G_K)=d/m$ just by taking the lower limit $n\to\infty$ in the equation above. If $K$ had strong Hausdorff dimension in $W_H$ it is clear that $G_K$ also has strong Hausdorff dimension in $W_H$. Note also that for any $d\in [0,1]$, the group $K$ may be chosen such that $K$ has strong Hausdorff dimension equal to $d$ in $W_H$ by \cite[Theorem A]{RestrictedSpectra} and \cref{theorem: strong self-similar}.

Now note that $1\ne H\le G_K$ is a normal subgroup of $G_K$. In fact, for any $h\in H$ and any $g\in D_m(K)$ we have $g^h=g$, i.e. $g$ and $h$ commute. Thus, we also get $h^g=h$ and so $H^{G_K}=H$, proving $H$ is normal in $G_K$. However $H$ is finite so $\mathrm{hdim}_{G_K}(H)=0$.

Lastly, note that $\mathrm{Rist}_{G_K}(1)=1$ so $G_K$ is not weakly branch. Indeed, we have $\mathrm{St}_{G_K}(1)=D_m(K)$ and therefore for any $v\in \mathcal{L}_1$ we get $\mathrm{rist}_{G_K}(v)=1$.

\subsection{Further properties of the groups $G_K$}

We compute the center of the group~$G_K\le_c W_H$ for any choice of $K$:

\begin{proposition}
\label{proposition: center}
    Let $1\ne K\le_c W_H$. Then $Z(G_K)\cong Z(H)\times Z(K)$. In particular, if $Z(H)\ne 1$ then $G_K$ has non-trivial center.
\end{proposition}
\begin{proof}
    We showed above that $H$ is normal in $G_K$ and clearly $H\cap D_m(K)=1$; thus $G_K=H\times D_m(K)$. Hence $Z(G_K)=Z(H)\times Z(D_m(K))\cong Z(H)\times Z(K)$.
\end{proof}

Now note that if $K\le W_H$ is branch then for $v\in \mathcal{L}_1$, the composition of the quotient map $\pi_H:G_K\to G_K/H\cong D_m(K)$ with the projection $\varphi_v:D_m(K)\to K\le W_H$ given by $\varphi_v(g):=g|_v$ defines a branch action of $G_K$ on $T$. Indeed, since $D_m(K)$ is diagonal, one has that $\ker (\varphi_v\circ \pi_H)=H$ and $\mathrm{Im}(\varphi_v\circ \pi_H)=K$ which is a branch group.

Finally, note that branch groups have trivial center by \cite[Theorem 2]{NewHorizonsGrigorchuk}. Therefore, if $K$ is branch we get $Z(G_K)\cong Z(H)$. Hence if $Z(H)<H$, i.e. if $H$ is not abelian, then the kernel of the above natural branch action of $G_K$ on $T$ is not central, asnwering \cref{Question: question 2 handbook} in the negative.

\section{Perfectness of the Hausdorff dimension}
\label{section: perfect hdim}

We now generalize \cite[Theorem 5]{AbertVirag} from $W_p$ to $W_H$ for any $H\le \mathrm{Sym}(m)$. Note that the only property of $W_H$ we shall use in the proof is that the logarithmic orders of the congruence quotients of $W_H$ grow exponentially on the level.

\begin{remark}
\label{remark: av gap}
    The original proof of Abért and Virág for $W_p$ in \cite{AbertVirag} claimed to be using a result of Kovács and Newman in \cite{KovacsNewman}, namely that for a transitive $p$-group $P\le \mathrm{Sym}(m)$, there exists a universal constant $K>0$ such that
    $$d(P)\le K\cdot\frac{m}{\sqrt{\log m}},$$
    where $d(P)$ is the minimal number of generators of $P$. However, this bound does not imply the bound 
    $$\log_p|P:P'|\le  K\cdot\frac{m}{\sqrt{\log m}},$$
    which is the bound that Abért and Virág claimed from this result (one needs to multiply the logarithmic exponent which is not bounded for the action on the tree and their proof fails). In fact, the above sharper bound was conjectured by Kovács and Praeger in \cite{KovacsPraeger} and remained open for over thirty years until it was finally proven recently by Lucchini, Sabatini and Spiga in \cite[Theorem 1]{Abelianization}. The best available bound for $\log_p|P:P'|$ when \cite{AbertVirag} was published was $m/p$, proved in \cite{KovacsPraeger} and in \cite{Bound} independently, which is not good enough for the proof in \cite[Theorem 5]{AbertVirag}.
\end{remark}

We follow mainly the proof of Abért and Virág in \cite[Theorem 5]{AbertVirag}. However, we simplify their approach and show that we do not need to separately consider orbits of length at least $\ell$ for each $\ell\ge 2$. This case is precisely the case in their partition that makes full use of the sublinear bound conjectured in \cite{KovacsPraeger} and recently proved in \cite{Abelianization}. For our proof it will suffice to consider the following bound in \cite{Bound} (see also \cite{KovacsPraeger}), which is easier to prove:

\begin{theorem}[{see \cite[Theorem 2]{Bound}}]
    \label{theorem: easy bound}
    Let $X$ be a finite set and $G\le \mathrm{Sym}(X)$. Then
    $$|G:G'|\le 2^{\#X-1}.$$
\end{theorem}

In particular by taking logarithms in the inequality in \cref{theorem: easy bound} we obtain
\begin{align}
    \label{align: upper bound abelianization}
    \log|G:G'|\le \# X-1\le \#X.
\end{align}
Furthermore, if $X$ is decomposed as a finite disjoint union of $G$-invariant subsets $X_1,\dotsc,X_r$ then it is clear that
    \begin{align}
        \label{align: invariant decomposition}
        \log|G:G'|\le \sum_{i=1}^r \log|G^{X_i}:(G^{X_i})'|,
    \end{align}
    where $G^{X_i}$ is the action of $G$ on $X_i$ for $1\le i\le r$.

\begin{proof}[Proof of \cref{theorem: hdim of commutators}]
    It is enough to show that
    $$\lim_{n\to\infty}\frac{\log|G:G'\mathrm{St}_G(n)|}{m^n}=0.$$
    Note that, if we write $G_n:=G/\mathrm{St}_G(n)$ for every $n\ge 1$, then
    $$\frac{G}{G'\mathrm{St}_G(n)}\cong \frac{G_n}{G_n'}.$$
    Hence
    $$\log|G:G'\mathrm{St}_G(n)|=\log|G_n:G_n'|,$$
    and we only need to prove that $\lim_{n\to\infty}\log|G_n:G_n'|m^{-n}=0$.


    Now, let us consider $G_n$ and its action on the $n$th level of $T$. Then this action may be split into disjoint orbits which are $G_n$-invariant. Let us write $\mathcal{O}_n$ for the set of orbits at level $n$ and let us fix any $k\ge 1$. We write $A_n$ for the orbits in $\mathcal{O}_n$ that are obtained from the orbits in $\mathcal{O}_{n-k}$ which branch in exactly $m^k$ distinct orbits at level $n$. We further set $B_n:=\mathcal{O}_n\setminus A_n$. Note that the orbits in $B_n$ are coming from orbits in $\mathcal{O}_{n-k}$ which branch in at most $m^k/2$ distinct orbits at level $n$. By \cref{align: invariant decomposition}, we have
    $$\log|G_n:G_n'|\le \log|G_n^{A_n}:(G_n^{A_n})'|+\log|G_n^{B_n}:(G_n^{B_n})'|.$$
    
    Note that the action of $G_n$ on $A_n$ is completely determined by the action of $G_{n-k}$ on the orbits formed by the $k$-predecessors of the orbits in $A_n$, namely $P_{n-k}\subseteq \mathcal{O}_{n-k}$. Thus by \cref{align: upper bound abelianization} we obtain
    \begin{align*}
        \log|G_n^{A_n}:(G_n^{A_n})'|&=\log|G_{n-k}^{P_{n-k}}:(G_{n-k}^{P_{n-k}})'|\le \# P_{n-k}\le \# \mathcal{O}_{n-k}\le m^{n-k}.
    \end{align*}

    Now, since $T$ is the $m$-adic tree $\#\mathcal{O}_{n+k}\le m^k\#\mathcal{O}_n$. In fact, the bound is attained precisely when each orbit in $\mathcal{O}_n$ branches in exactly $m^k$ distinct orbits at level $n+k$, i.e. when $\mathcal{O}_{n+k}=A_{n+k}$. Since $\mathcal{O}_{n+k}:=A_{n+k}\sqcup B_{n+k}$ and $B_{n+k}$ consists of those orbits in $\mathcal{O}_{n+k}$ that are obtained from orbits in $\mathcal{O}_n$ which branch in at most $m^k/2$ distinct orbits at level $n+k$, we obtain the inequality
    $$\#A_{n+k}+2\#B_{n+k}\le m^k\#\mathcal{O}_n,$$
    as each orbit in $B_{n+k}$ is being counted at least twice by the upper bound $ m^k\#\mathcal{O}_n$. Hence
    $$\#B_{n+k}=\#A_{n+k}+2\#B_{n+k}-(\#A_{n+k}+\#B_{n+k})\le m^k\#\mathcal{O}_n-\#\mathcal{O}_{n+k}.$$
    If $u$ is a vertex at the $n$th level and $v$ a descendant of $u$ at level $r\ge n$, then the $G_r$-orbit of $v$ is at least as large as the $G_n$-orbit of $u$, i.e. orbits grow as one goes down in the tree. Then, the sequence $\#\mathcal{O}_n/m^n$ is monotone so it converges, and we get
    $$0\le \limsup_{n\to\infty}\frac{\#B_{n+k}}{m^{n+k}}\le \lim_{n\to\infty}\frac{\#\mathcal{O}_n}{m^n}-\lim_{n\to\infty}\frac{\#\mathcal{O}_{n+k}}{m^{n+k}}=0.$$
    In other words
    $$\lim_{n\to \infty}\frac{\#B_n}{m^n}=0.$$
    Now, again by \cref{align: upper bound abelianization}, we get the upper bound
    $$\log|G_n^{B_n}:(G_n^{B_n})'|\le \#B_n.$$
    Putting everything together yields
    \begin{align*}
        \frac{\log|G_n:G_n'|}{m^n}&\le \frac{1}{m^n}\left(\log|G_n^{A_n}:(G_n^{A_n})'|+\log|G_n^{B_n}:(G_n^{B_n})'|\right)\le \frac{1}{m^k}+ \frac{\#B_n}{m^n},
    \end{align*}
    and taking the upper limit as $n\to\infty$ we get
    \begin{align}
        \label{align: final inequality orbits}
        \limsup_{n\to \infty}\frac{\log|G_n:G_n'|}{m^n}\le  \frac{1}{m^k}+ \limsup_{n\to \infty}\frac{\#B_n}{m^n}=\frac{1}{m^k}.
    \end{align}
    Since the inequality in \cref{align: final inequality orbits} holds for any $k\ge 1$ the result follows.
\end{proof}

\section{The self-similar case}
\label{section: self-similar setting}

Now we restrict our attention to the family of self-similar groups. We first show that the rigid level stabilizers are 1-dimensional in $G$ if $G$ is a self-similar closed subgroup with positive Hausdorff dimension in $W_H$. This fact together with \cref{theorem: hdim of commutators} and some key results in \cite{FiniteSpectrum} yield \cref{Theorem: self-similar spectra}.

\subsection{Hausdorff dimension of rigid stabilizers}

First let us give a lower bound on the Hausdorff dimension of rigid vertex stabilizers. Recall that for a vertex $v\in T$ we write $l(v)$ for the level at which $v$ lies on.

\begin{lemma}
\label{lemma: rists dimension}
    Let $G\le_c W_H$ be a closed self-similar subgroup with positive Hausdorff dimension in $W_H$. Then for any vertex $v\in T$ the subgroup $\mathrm{rist}_G(v)$ has Hausdorff dimension at least $m^{-l(v)}$ in $G$. In particular $G$ is weakly branch.
\end{lemma}
\begin{proof}
    Let us denote $d:=\mathrm{hdim}_{W_H}(G)>0$ to simplify notation. Let us fix a level $k\ge 1$. Since $G\le W_H$ is self-similar, it has strong Hausdorff dimension in $W_H$ by \cref{theorem: strong self-similar}. The same holds for any finite-index subgroup of $G$, in particular for $\mathrm{St}_G(k)$. Hence, for any $\epsilon>0$ there exists $N_\epsilon\ge k$ such that 
    \begin{align}
    \label{align: inequalities log indices}
        d-\epsilon&\le \frac{\log|\mathrm{St}_G(k):\mathrm{St}_G(n)|}{\log|W_H:\mathrm{St}_{W_H}(n)|}\le \frac{\log|G:\mathrm{St}_G(n)|}{\log|W_H:\mathrm{St}_{W_H}(n)|}\le d+\epsilon
    \end{align}
    for all $n\ge N_\epsilon$. 
    
    Since $G$ is self-similar, the monomorphism $\psi_k:\mathrm{St}_G(k)\to G\times\overset{m^k}{\ldots}\times G$ induces a monomorphism
    \begin{align}
    \label{align: embedding ss}
        \frac{\mathrm{St}_G(k)}{\mathrm{St}_G(n)}&\hookrightarrow \frac{G}{\mathrm{St}_G(n-k)}\times \overset{m^k}{\dotsb}\times \frac{G}{\mathrm{St}_G(n-k)}
    \end{align}
    for every $n\ge k$. In the following, we shall assume $\psi_k$ is the induced monomorphism in \cref{align: embedding ss}. We shall also assume in the following that $n\ge N_\epsilon+k$.

    For a vertex $v$ at the $k$th level of $T$, let us denote by $\psi_{\lnot v}$ the composition $\pi_{\lnot v} \circ \psi_k$, where $\pi_{\lnot v}$ is the natural projection on all the coordinates of the image of $\psi_k$ but the one corresponding to the vertex $v$. Then by \textcolor{teal}{Equations (}\ref{align: inequalities log indices}\textcolor{teal}{)} and \textcolor{teal}{(}\ref{align: embedding ss}\textcolor{teal}{)}  we get
    \begin{align}
        \label{align: size of image}
        |\mathrm{Im}~\psi_{\lnot v}|\le |G:\mathrm{St}_G(n-k)|^{m^k-1}\le |W_H:\mathrm{St}_{W_H}(n-k)|^{(m^k-1)(d+\epsilon)}.
    \end{align}
    Hence, as $\psi_{\lnot v}$ is a group homomorphism, by the first isomorphism theorem and \textcolor{teal}{Equations (}\ref{align: inequalities log indices}\textcolor{teal}{)} and \textcolor{teal}{(}\ref{align: size of image}\textcolor{teal}{)}, the order of the kernel of $\psi_{\lnot v}$ is at least
    \begin{align*}
        |\ker \psi_{\lnot v}|=\frac{|\mathrm{St}_G(k):\mathrm{St}_G(n)|}{|\mathrm{Im}~\psi_{\lnot v}|}&\ge \frac{|W_H:\mathrm{St}_{W_H}(n)|^{(d-\epsilon)}}{|W_H:\mathrm{St}_{W_H}(n-k)|^{(m^k-1)(d+\epsilon)}}\\
        &\ge\frac{|W_H:\mathrm{St}_{W_H}(n-k)|^{m^k(d-\epsilon)}}{|W_H:\mathrm{St}_{W_H}(n-k)|^{(m^k-1)(d+\epsilon)}}\\
        &=|W_H:\mathrm{St}_{W_H}(n-k)|^{d+\epsilon(1-2m^k)}.
    \end{align*}
    
    Now, note that by the definition of the rigid vertex stabilizer we have
    $$|\mathrm{rist}_G(v):\mathrm{St}_{\mathrm{rist}_G(v)}(n)|=|\ker \psi_{\lnot v}|,$$
    Therefore, putting all together we obtain
    \begin{align*}
    \mathrm{hdim}&_G(\mathrm{rist}_G(v))=\liminf_{n\to\infty}\frac{\log|\mathrm{rist}_G(v):\mathrm{St}_{\mathrm{rist}_G(v)}(n)|}{\log|G:\mathrm{St}_G(n)|}\\
    &\ge \liminf_{n\to\infty}\frac{\big(d+\epsilon(1-2m^k)\big)\log|W_H:\mathrm{St}_{W_H}(n-k)|}{\log|G:\mathrm{St}_G(n)|}\\
    &=\big(d+\epsilon(1-2m^k)\big)\cdot\liminf_{n\to\infty}\frac{m^{-k}\big(-\log|H\wr\overset{k}{\dotsb}\wr H|+\log|W_H:\mathrm{St}_{W_H}(n)|\big)}{\log|G:\mathrm{St}_G(n)|}\\
    &=\frac{\big(d+\epsilon(1-2m^k)\big)}{m^k}\cdot\lim_{n\to\infty}\frac{\log|W_H:\mathrm{St}_{W_H}(n)|}{\log|G:\mathrm{St}_G(n)|}\\
    &=\frac{\big(d+\epsilon(1-2m^k)\big)}{m^kd}.
    \end{align*}
    Since the above holds for any $\epsilon>0$ we get
\begin{align*}
    \mathrm{hdim}_G(\mathrm{rist}_G(v))&\ge \frac{1}{m^{k}}.\qedhere
\end{align*}
\end{proof}

\begin{remark}
    Note that if $\mathrm{hdim}_{W_H}(G)=1$, then one does not need self-similarity to obtain the inequality in \cref{align: size of image} and one recovers \cref{theorem: abert and virag normal}. However, as soon as $d<1$, the proof fails if self-similarity is not assumed. However, the counterexamples in \cref{Theorem: counterexample} have dimension at most $1/m$, so it would be interesting to see whether there exist counterexamples of dimension arbitrary close to~1 or whether self-similarity may be dropped from \cref{Theorem: self-similar spectra} for dimensions close enough to 1.
\end{remark}

\begin{lemma}
    \label{lemma: rists lower bound for strong hdim}
    Let $G\le_c W_H$ be a closed subgroup. Assume further that for any vertex $v\in T$ we have $\mathrm{hdim}_G(\mathrm{rist}_G(v))\ge m^{-l(v)}$. Then $\mathrm{Rist}_G(k)$ has Hausdorff dimension~1 in~$G$ for every $k\ge 1$.
\end{lemma}
\begin{proof}
    Since $\mathrm{hdim}_G(\mathrm{Rist}_G(k))\le 1$, it is enough to show that $\mathrm{hdim}_G(\mathrm{Rist}_G(k))\ge 1$. Then the result follows from the equality
    $$\log|\mathrm{Rist}_G(k):\mathrm{St}_{\mathrm{Rist}_G(k)}(n)|=\sum_{v\in \mathcal{L}_k}\log|\mathrm{rist}_G(v):\mathrm{St}_{\mathrm{rist}_G(v)}(n)|$$
    as
    \begin{align*}
        \mathrm{hdim}_G(\mathrm{Rist}_G(k))&=\liminf_{n\to\infty}\frac{\log|\mathrm{Rist}_G(k):\mathrm{St}_{\mathrm{Rist}_G(k)}(n)|}{\log|G:\mathrm{St}_G(n)|}\\
        &\ge \sum_{v\in \mathcal{L}_k} \liminf_{n\to\infty}\frac{\log|\mathrm{rist}_G(v):\mathrm{St}_{\mathrm{rist}_G(v)}(n)|}{\log|G:\mathrm{St}_G(n)|}\\
        &=\sum_{v\in \mathcal{L}_k}\mathrm{hdim}_G(\mathrm{rist}_G(v))\\
        &\ge 1.\qedhere
    \end{align*}
\end{proof}

\subsection{Proof of the main results}

We can now answer \textcolor{teal}{Questions} \ref{Question: abert and virag normal} and \ref{Question: positive dim implies weakly branch} in the positive for self-similar positive-dimensional groups:

\begin{proposition}
\label{proposition: self-similar}
    Let $G\le W_H$ be a closed level-transitive self-similar group with positive Hausdorff dimension in $W_H$. Then $G$ is weakly branch and every non-trivial normal closed subgroup $N\trianglelefteq_c G$ has Hausdorff dimension 1 in $W_H$.
\end{proposition}
\begin{proof}
    The group $G$ is weakly branch directly by \cref{lemma: rists dimension}. Now, any non-trivial normal subgroup $N$ contains the commutator subgroup of some rigid level stabilizer $\mathrm{Rist}_G(n)$ (this is a well-known fact, see for instance \cite[Lemma 4]{pro-c}). Therefore, by \cref{lemma: rists lower bound for strong hdim} we have that $\mathrm{hdim}_{G}(\mathrm{Rist}_G(n))=1$ and thus
    $$\mathrm{hdim}_{G}(N)\ge \mathrm{hdim}_{G}(\mathrm{Rist}_G(n)')=\mathrm{hdim}_{G}(\mathrm{Rist}_G(n))=1,$$
    where the first inequality follows from monotonicity of the Hausdorff dimension and the first equality by \cref{theorem: hdim of commutators} and \cref{lemma: strong dimension then product}, as 1-dimensional subgroups have strong Hausdorff dimension.
\end{proof}

By \cref{theorem: hdim of commutators} we may drop the assumption (ii) in \cite[Theorem 3.5]{FiniteSpectrum} for the usual stabilizer filtration:

\begin{theorem}[{see {\cite[Theorem 3.5]{FiniteSpectrum}}}]
\label{theorem: full spectra jone and oihana}
    Let $G\le_c W_H$ be a closed and level-transitive subgroup such that for all $n\ge 1$ we have $\mathrm{hdim}_{G}(\mathrm{Rist}_G(n))=1$.
    Then the Hausdorff spectrum of $G$ is given by
        $$\mathrm{hspec}(G)=[0,1],$$
    and the normal Hausdorff spectrum of $G$ is given by
    $$\mathrm{hspec}_{\trianglelefteq}(G)=\{0,1\}.$$
\end{theorem}

\begin{proof}[Proof of \cref{Theorem: self-similar spectra}]
    The result follows from \textcolor{teal}{Lemmata} \ref{lemma: rists dimension} and \ref{lemma: rists lower bound for strong hdim}, \cref{proposition: self-similar} and  \cref{theorem: full spectra jone and oihana}.
\end{proof}



\bibliographystyle{unsrt}

\end{document}